\documentclass[12pt,reqno]{amsart}

\usepackage{amsfonts, amsthm, amsmath,bm}
\allowdisplaybreaks[4]

\usepackage{rotating}

\usepackage{tikz}

\usepackage{stmaryrd}
\usepackage{multirow}

\usepackage{gensymb}
\usepackage{bm}
\usepackage{graphics}

\usepackage{amssymb}

\usepackage{amscd}

\usepackage[latin2]{inputenc}

\usepackage{t1enc}

\usepackage[mathscr]{eucal}

\usepackage{indentfirst}

\usepackage{graphicx}

\usepackage{graphics}

\usepackage{pict2e}

\usepackage{mathrsfs}

\usepackage{enumerate}

\usetikzlibrary{matrix,arrows}
\usetikzlibrary{positioning}
\usetikzlibrary{fit}
\usetikzlibrary{patterns}

\usepackage{pgfplots}

\usepackage[pagebackref]{hyperref}
\hypersetup{colorlinks=true}
\usepackage{cite}
\usepackage{color}
\usepackage{epic}
\usepackage{hyperref} 
\numberwithin{equation}{section}
\theoremstyle{plain}

\usepackage{titlesec}
\titleformat{\section}
  {\normalfont\large\bfseries}{\thesection.}{.8ex}{} 
\titleformat{\subsection}
  {\normalfont\bfseries}{\thesubsection.}{.8ex}{} 

\usetikzlibrary{shapes}

\tikzstyle{pathdefault}=[draw, line width=1, solid, color=black]
\tikzstyle{nodedefault}=[circle, inner sep=1.5, fill=black]
\tikzstyle{empty}=[]
\tikzstyle{nodeellipsis}=[circle, inner sep=0.5, fill=black]
\tikzstyle{pathcolor1}=[draw, line width=1.3, densely dashed, color=red]
\tikzstyle{pathcolor2}=[draw, line width=1.6, densely dotted, color=blue]
\tikzstyle{pathcolorlight}=[draw, line width=1, dotted, color=lightgray]

\tikzstyle{arbpathcolor0}=[line width=1, dashdotted, color=black]
\tikzstyle{arbpathcolor1}=[line width=1, densely dashed, color=red]
\tikzstyle{arbpathdefault}=[line width=1, densely dotted, color=blue]

\newcounter{id}
\newcommand{\drawlinedotswithstyle}[4]{
 \def\x{{#3}}
 \def\y{{#4}}
 \tikzstyle{thispathstyle}=[#1]
 \tikzstyle{thisnodestyle}=[#2]
 \setcounter{id}{-1} 
 \foreach \j in {#3}{\stepcounter{id}} 
 \foreach \i in {1,...,\the\value{id}}{  
  \path[thispathstyle] (\x[\i],\y[\i]) --(\x[\i-1],\y[\i-1]); 
 }
 \foreach \i in {1,...,\the\value{id}}{  
  \node[thisnodestyle] at (\x[\i],\y[\i]) {}; 
 }
 \node[thisnodestyle] at (\x[0],\y[0]) {}; 
}

\definecolor{mhcblue}{HTML}{0077CC} 
\definecolor{davidsonred}{HTML}{AC1A2F} 

\definecolor{green}{RGB}{0, 180, 0}
\definecolor{yellow}{RGB}{180, 180, 0}

\newtheorem{theorem}{Theorem}[section]

\newtheorem{lemma}[theorem]{Lemma}

\newtheorem{corollary}[theorem]{Corollary}

\theoremstyle{definition}

\newtheorem{example}[theorem]{Example}

\newtheorem{conj}[theorem]{Conjecture}

\newtheorem{remark}[theorem]{Remark}

\newtheorem{?}[theorem]{Problem}

\newcommand{\dbrac}[1]{{\llbracket#1\rrbracket}} 	
\newcommand{\boks}[2]{({#1, #2})}   

\newcommand{\pattern}[4]{										
	\raisebox{0.6ex}{
		\begin{tikzpicture}[scale=0.35, baseline=(current bounding box.center), #1]
		\foreach \x/\y in {#4}		\fill[pattern=north east lines] (\x,\y) rectangle +(1,1);
		\draw (0.01,0.01) grid (#2+0.99,#2+0.99);
		\foreach \x/\y in {#3}		\filldraw (\x,\y) circle (6pt);
		\end{tikzpicture}}
}

\definecolor{red}{rgb}{1,0,0}

\def\S{\mathfrak{S}}

\def\exc{\operatorname{exc}}

\def\inv{\mathrm{inv}}

\def\ca-b{31\textsf{-}2}
\def\b-ca{2\textsf{-}31}

\def\r{\mathrm{r}}

\def\Nr{\mathrm{Nr}}

\def\erec{\mathrm{erec}}
\def\rec{\mathrm{rec}}
\def\arec{\mathrm{arec}}
\def\earec{\mathrm{earec}}
\def\rar{\mathrm{rar}}

\def\AREC{\mathrm{AREC}}

\begin{document}
\title[Equidistributions of mesh patterns of length two]
{Equidistributions of mesh patterns of length two and Kitaev and Zhang's conjectures}
\author[B. Han]{Bin Han}
\address[Bin Han]{Department of Mathematics, Bar-Ilan University, Ramat-Gan 52900, Israel.}
\email{han@math.biu.ac.il, han.combin@hotmail.com}

\author[J. Zeng]{Jiang Zeng}
\address[Jiang Zeng]{Univ Lyon, Universit\'e Claude Bernard Lyon 1, CNRS UMR 5208, Institut Camille Jordan, 43 blvd. du 11 novembre 1918, F-69622 Villeurbanne cedex, France}
\email{zeng@math.univ-lyon1.fr}

\date{\today}

\begin{abstract}
A systematic study of {\em avoidance} of mesh patterns of length 2 was conducted by Hilmarsson \emph{et al.} in 2015. In a recent paper
Kitaev and Zhang  examined the  distribution of 
the aforementioned 
patterns.    The aim of this paper is to  prove more equidistributions  of mesh pattern  and            confirm Kitaev and Zhang's four conjectures  by constructing two involutions on permutations.

%
%

\end{abstract}



\maketitle



\section{Introduction}
Patterns in permutations and words have implicitly appeared in the mathematics literature for over a century, but interest in them has blown up  in the past four  decades (see~\cite{Kit,Stein,  BrCl, Hilmarsson2015Wilf, KZ19} and references therein),  and the research of this area continues to increase gradually. 

Let $S_n$ be the set of all permutations of length $n$. A (classical permutation) pattern is a permutation $\tau\in S_n$.  We could draw the pattern $231\in S_3$ as follows, where the horizontal lines represent the values and the vertical lines denote the positions in the pattern.
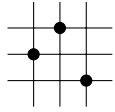
\[
\pattern{scale=1}{3}{1/2,2/3,3/1}{}
\]
To study the explicit expansions for certain permutation statistics as, possibly infinite, linear combinations of (classical) permutation patterns, Br\"and\'en and Claesson \cite{BrCl}  first introduced the notion of a {\em mesh pattern}, which generalize several classes of patterns.

A pair $(\tau,R)$, where $\tau$ is a permutation of length $k$ and $R$ is a subset of $\dbrac{0,k} \times \dbrac{0,k}$, where
$\dbrac{0,k}$ denotes the interval of the integers from $0$ to $k$, is a
\emph{mesh pattern} of length $k$.
Let $\boks{i}{j}$ denote the box whose corners have coordinates $(i,j), (i,j+1),
(i+1,j+1)$, and $(i+1,j)$. 
Mesh patterns can be depicted by shading the boxes in $R$. A mesh pattern with $\tau=231$ and $R = \{\boks{1}{2},\boks{2}{1}\}$ is drawn as follows.
\[
\pattern{scale=1}{3}{1/2,2/3,3/1}{1/2, 2/1}
\]
For example, the permutation $346512$ depicted in the following picture contains the mesh pattern $(231, \{\boks{1}{2},\boks{2}{1}\})$ since the subsequence $462$ forms the classical pattern $231$ and there are no points in the shaded
areas.
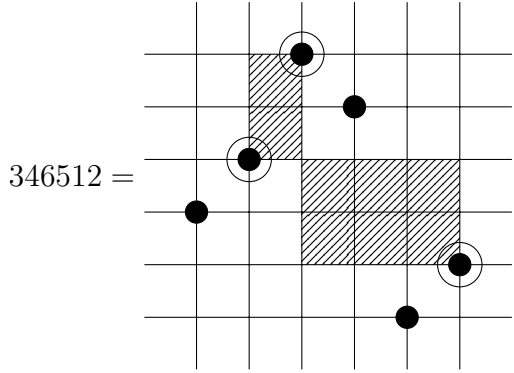
\[
346512 =
 \begin{tikzpicture}[scale=0.7, baseline=(current bounding box.center)]
    \draw (0.01,0.01) grid (6+0.99,6+0.99);
    \fill[pattern=north east lines] (2,4) rectangle +(1,2);
  \fill[pattern=north east lines] (3,2) rectangle +(3,2);
  \filldraw (1,3) circle (6pt);
  \filldraw (2,4) circle (6pt);
  \filldraw (3,6) circle (6pt);
  \filldraw (4,5) circle (6pt);
  \filldraw (5,1) circle (6pt);
  \filldraw (6,2) circle (6pt);
  \draw (2,4) circle (12pt);
  \draw (3,6) circle (12pt);
  \draw (6,2) circle (12pt);
 \end{tikzpicture}
\]

The mesh patterns and their generalizations were studied in many papers; e.g.\ see \cite{AKV,Borie,Hilmarsson2015Wilf,JKR,KL,KR1,KRT,KZ19,T1,T2}. 
In the first systemically study of the mesh patterns avoidance, Hilmarsson et al.~\cite{Hilmarsson2015Wilf}  solved 25 out of 65 non-equivalent {\em avoidance} cases of patterns of length 2.
In a recent paper~\cite{KZ19}, Kitaev and Zhang further studied the distributions of mesh patterns considered in \cite{Hilmarsson2015Wilf} by giving 27 distribution results  
see \cite[Table~1]{KZ19}.
Moreover, for the unsolved case, they gave an equidistribution result and  conjectured 6 more equidistributions (see Table~\ref{tab-2}). In this paper, we prove 3 conjectured equidistributions and 2 more equidistribuions (see Table~\ref{tab-3}) by constructing two involutions.

\begin{table}[htbp]
{
		\renewcommand{\arraystretch}{1.3}
	\begin{center}
	\begin{tabular}{|c|c|c|c||c|c|c|}
		\hline
		&  	{Nr.\ } & {Repr.\ $p$}  & {Ref.}  &  {Nr.\ } & {Repr.\ $p$}  & {Ref.}  
		\\[6pt]
		\hline \hline
		{proved} 
		& 48 & $\pattern{scale = 0.6}{2}{1/1,2/2}{0/1,1/2,0/0,2/1,2/2}$&\multirow{2}{*}{\cite[Theorem~5.1]{KZ19}}&&&
		 \\[6pt]	\cline{2-3} 
     {equidistributions}   & 49 &	$\pattern{scale = 0.6}{2}{1/1,2/2}{0/1,1/2,0/0,1/1,2/0}$   &	 &&&	\\[6pt]	\hline
   	 & 	23 &	$\pattern{scale = 0.6}{2}{1/1,2/2}{0/0,0/2,1/0,1/1,1/2}$  & \multirow{2}{*}{Theorem~\ref{main2}}  &53 &	$\pattern{scale = 0.6}{2}{1/1,2/2}{0/1,1/2,0/0,2/1}$  & \multirow{2}{*}{Theorem~\ref{main1}}
		\\[6pt]
	\multirow{2}{*}{conjectured} & 	24 &	$\pattern{scale = 0.6}{2}{1/1,2/2}{0/0,0/1,1/0,1/1,1/2}$ &&	54 &  $\pattern{scale = 0.6}{2}{1/1,2/2}{0/1,0/0,1/1,2/2}$   &			\\[6pt] \cline{2-7} 
	\multirow{3}{*}{equidistributions} &	 \multirow{2}{*}{48} & \multirow{2}{*}{$\pattern{scale = 0.6}{2}{1/1,2/2}{0/1,1/2,0/0,2/1,2/2}$}	  & 	
	&	57 & $\pattern{scale = 0.6}{2}{1/1,2/2}{0/1,1/2,1/1,2/0}$	&	 \multirow{2}{*}{N/A}	\\[6pt]		
	 	& \multirow{2}{*}{49} &   \multirow{2}{*}{$\pattern{scale = 0.6}{2}{1/1,2/2}{0/1,1/2,0/0,1/1,2/0}$}   &	 \multirow{2}{*}{Theorem~\ref{main1}  and}
		&	58 &  $\pattern{scale = 0.6}{2}{1/1,2/2}{0/1,1/0,1/1,2/2}$ &	 	\\[6pt]	\cline{5-7}
	&\multirow{2}{*}{50} &  \multirow{2}{*}{$\pattern{scale = 0.6}{2}{1/1,2/2}{0/1,1/2,0/0,1/1,2/2}$} &\multirow{2}{*}{\cite[Theorem~5.1]{KZ19}} & 61 & $\pattern{scale = 0.6}{2}{1/1,2/2}{0/1,1/2,0/0,2/0}$	&  \multirow{2}{*}{N/A}		\\[6pt]
		&  && &	62 & $\pattern{scale = 0.6}{2}{1/1,2/2}{0/1,1/0,0/0,2/2}$ &		\\[6pt]	\hline
	\end{tabular}
	\end{center}
}
\medskip
	\caption{Equidistributions for which enumeration is unknown. Pattern's numbers are adopted from  \cite{Hilmarsson2015Wilf,KZ19}}
 	\label{tab-2}
\end{table}

 \begin{table}[htbp]
{
		\renewcommand{\arraystretch}{1.3}
	\begin{center}
	\begin{tabular}{|c|c|c|c||c|c|c|}
		\hline
		&  	{Nr.\ } & {Repr.\ $p$}  & {Ref.}  &  {Nr.\ } & {Repr.\ $p$}  & {Ref.}  
		\\[6pt]
		\hline \hline
		{proved} 
		& $1^*$ & $\pattern{scale=0.5}{2}{1/2,2/1}{1/0,2/0,2/2}$&\multirow{2}{*}{Theorem~\ref{main2}}& $3^*$&  $\pattern{scale=0.5}{2}{1/2,2/1}{1/0, 1/1,2/0,2/2}$ & $\multirow{2}{*}{Theorem~\ref{main2}}$
		 \\[6pt]	
     {equidistributions}   &  $2^*$ &	$\pattern{scale=0.5}{2}{1/2,2/1}{2/1,1/0,2/0}$  & 	 & $4^*$ &$\pattern{scale=0.5}{2}{1/2,2/1}{2/1,1/0,2/0,1/2}$  &	\\[6pt] 	\hline
	\end{tabular}
	\end{center}
}
\medskip
	\caption{More proved equidistributions. Pattern's numbers are not considered in \cite{Hilmarsson2015Wilf,KZ19}}
 	\label{tab-3}
\end{table} 
%

 For a pattern $p$ and a permutation $\pi$, we let $p(\pi)$ denote the number of occurrences of $p$ in $\pi$. 
Kitaev and Zhang~\cite[Conjecture~6.1]{KZ19} conjectured a Stieltjes continued fraction formula for 
the distribution of pattern Nr.\ $3$ = $\pattern{scale = 0.6}{2}{1/1,2/2}{0/0,0/1,1/2}$ (see \cite[$A200545$]{OEIS}), which 
is equivalent to the following identity.

\begin{conj}\cite[Conjecture~6.1]{KZ19}\label{conj1} 
 We have 
 \begin{subequations}
\begin{align}\label{cf:sz}
\sum_{n\geq 0}t^n
 \sum_{\pi\in S_n}y^{\pattern{scale=0.5}{2}{1/1,2/2}{0/0,0/1,1/2}(\pi)}\;=\;
 \cfrac{1}
 {1-\cfrac{\alpha_1 t}
 {1-\cfrac{\alpha_2 t}
 {1-\cdots}}}
\end{align}
\text{with coefficients} 
\begin{align}
\alpha_{2k-1}=k,\quad 
\alpha_{2k}=y+k-1.
\end{align}
\end{subequations}
\end{conj}

Presenting their  conjecture in this way, we notice that 
the S-continued fraction \eqref{cf:sz} 
 appears in a recent paper of Sokal and Zeng~\cite{SZ20}.
Let~us reformulate the relevant permutation statistics in \cite{SZ20} in terms of mesh patterns.
Given a permutation $\pi \in S_n$,
an index $i \in [n]$ (or a value $\pi(i) \in [n]$) is called 
\begin{itemize}
\item  an {\em excedance} if $\pi(i)>i$;
   
 \item an {\em inversion} if $\pi(j)>\pi(i)$ for $1\leq j<i$; in other words, an inversion of $\pi$ is one occurrence of pattern $\pattern{scale=0.5}{2}{1/2,2/1}{}$ of $\pi$;

\item a {\em record}\/ (rec) (or {\em left-to-right maximum}\/)
         if $\pi(j) < \pi(i)$ for all $j < i$
      [note in particular that the index 1 is always a record
       and that the value $n$ is always a record];
in other words, a record of $\pi$ is one occurrence of pattern $\pattern{scale=0.8}{1}{1/1}{0/1}$ of  $\pi$;

   \item an {\em antirecord}\/ (arec) (or {\em right-to-left minimum}\/)
         if $\pi(j) > \pi(i)$ for all $j > i$
      [note in particular that the index $n$ is always an antirecord
       and that the value 1 is always an antirecord];
in other words, an antirecord of $\pi$ is one occurrence of pattern $\pattern{scale=0.8}{1}{1/1}{1/0}$ of  $\pi$;      
       
   \item an {\em exclusive record}\/ (erec) if it is a record and not also
         an antirecord; in other words, an exclusive record of $\pi$ is one occurrence of pattern $\pattern{scale=0.5}{2}{1/2,2/1}{0/2,1/0,2/0}$ of  $\pi$, see \eqref{eq2:lemma1};       
         
   \item an {\em exclusive antirecord}\/ (earec) if it is an antirecord and not also
     a record; in other words, an exclusive antirecord of $\pi$ is one occurrence of pattern $\pattern{scale=0.5}{2}{1/2,2/1}{0/2,1/2,2/0}$ of  $\pi$, see \eqref{eq1:lemma1};     
     
 \item a {\em record-antirecord}\/ (rar) (or {\em pivot}\/)
     if it is both a record and an antirecord;
   in other words, a record-antirecord of $\pi$ is one occurrence of pattern  \pattern{scale=0.8}{1}{1/1}{0/1,1/0}  of  $\pi$.


\end{itemize}
We denote the number of excedances, records, antirecords, exclusive records, exclusive antirecords and record-antirecords in $\pi$
by $\exc(\pi)$, $\rec(\pi)$, $\arec(\pi)$, $\erec(\pi)$, $\earec(\pi)$ and $\rar(\pi)$,  respectively. 

Dumont and Kreweras~\cite{DK88} gave the joint distribution of $(\pattern{scale=0.8}{1}{1/1}{0/1},\pattern{scale=0.5}{2}{1/2,2/1}{0/2,1/2,2/0})$, Zeng \cite{Zeng89} gave the joint distribution of $(\pattern{scale=0.8}{1}{1/1}{0/1},\pattern{scale=0.5}{2}{1/2,2/1}{0/2,1/2,2/0}, \pattern{scale = 0.5}{2}{1/2,2/1}{})$. 
Recently Sokal and Zeng~\cite{SZ20} proved  much more general results.
 For example,
define 
the generating function of the generalized Eulerian polynomials
\begin{align}\label{poly:dkz}
F(x,y,z, v, q; t)=\sum_{n=0}^\infty t^n\sum_{\sigma\in \S_n}
x^{\arec(\sigma)}y^{\erec(\sigma)}z^{\rar(\sigma)}v^{\exc(\sigma)}q^{\inv(\sigma)}.
\end{align}
From \cite[Theorems~2.7 and 2.8]{SZ20} we derive   the following result.
\begin{theorem}\label{thm:1.2}
We have 
\begin{subequations}
\begin{align}\label{cf-SZ}
F(x,y,z, v, q; t)=
\frac{F(x,y,1, v,q; t)}{1+x(1-z)tF(x,y, 1, v,q; t)},
\end{align}
where
\begin{align}\label{cf-pattern}
F(x,y,1, v,q; t)=\cfrac{1}{1-\cfrac{\alpha_1 t}{1-\cfrac{\alpha_2 t}
{1-\cdots}}}
\end{align}
\text{with coefficients} 
\begin{align}
\alpha_{2k-1}&=q^{k-1}(x+q+q^2+\cdots +q^{k-1})\\
\alpha_{2k}&=q^{k}v(y+q+q^2+\cdots +q^{k-1}).
\end{align}
\end{subequations}
\end{theorem}
\begin{proof} This follows from \cite[Theorems~2.7 and 2.8]{SZ20}
by specializing the parameters.
We just indicate the appropriate specialisation  and refer the reader to \cite{SZ20}
for further details.
In the specialization (2.57) of \cite[Theorem~2.7]{SZ20},  if we   choose
$w_0=xz$ (instead of $w_0=x$ in \cite{SZ20}) and 
$$
y=qv,\quad u=1,\quad v_1=v_2=qv, \quad p_{+}=p_{-}=q,\quad  q_{+}=q_{-}=q^2, 
$$
then equation (2.52) of \cite{SZ20} reduces to
\begin{align}
F(x,y,z, v, q; t)&=\frac{1}{1-\gamma_0 t-
\cfrac{\beta_1 t}{1-\gamma_1 t-\cfrac{\beta_2 t}{1-\cdots}}}
\end{align}
with 
\begin{align}
\gamma_0=xz,\quad
\gamma_n=\alpha_{2n}+\alpha_{2n+1},\quad 
\beta_n=\alpha_{2n-1}\alpha_{2n}.
\end{align}
Therefore, the J-fraction formula can be written as (by contracting the 
S-fraction starting from the second line),
\begin{align}
F(x,y,z, v, q; t)&=\frac{1}{1+x(1-z)t-
\cfrac{\alpha_1 t}{1-\cfrac{\alpha_2 t}{1-\cdots}}},
\end{align}
which is equivalent to \eqref{cf-SZ}.
\end{proof}
\begin{remark}
\begin{itemize}
\item We can also prove  \eqref{cf-SZ} by following the same
steps in  the special case  as in \cite{KZ19}  and then derive  \eqref{cf-pattern} directly from \cite[Theorem 2.8]{SZ20}. 
\item The case $x=y=v=q=1$ of Theorem~\ref{thm:1.2} is Theorem~1.1 in \cite{KZ19}.
\item Since 
$(\arec, \inv)\pi=(\rec, \inv)\pi^{-1}$ we derive from
 \cite{Zeng89} that  
\begin{equation}
F(x,1,1, 1,q; t)=\sum_{n=0}^\infty x(x+q)\ldots (x+q+\cdots q^{n-1})t^n.
\end{equation}
\end{itemize}
\end{remark}


For $\pi=\pi(1)\ldots \pi(n)\in S_n$ we 
define the following three associated  permutations:
\begin{align}
\pi^{-1} &:=\pi^{-1}(1)\pi^{-1}(2)\cdots\pi^{-1}(n)\label{def:invpi}\\
\pi^\r &:=\pi(n)\cdots\pi(2)\pi(1)\label{def:revpi}\\
\pi^c &:=(n+1-\pi(1))(n+1-\pi(2))\cdots(n+1-\pi(n))\label{def:comlpi}
\end{align}
Obviously 
we have 
$$
\pattern{scale=0.5}{2}{1/2,2/1}{0/2,1/0,2/0} (\pi)=
\pattern{scale=0.5}{2}{1/1,2/2}{0/0,1/2,2/2} (\pi^c)=
\pattern{scale=0.5}{2}{1/2,2/1}{0/2,1/2,2/0} (\pi^{r\circ c})=
\pattern{scale=0.5}{2}{1/1,2/2}{0/0,1/0,2/2} (\pi^{r})
$$
and
\begin{align*}
&\pattern{scale=0.5}{2}{1/1,2/2}{0/0,0/1,1/2} (\pi)=
\pattern{scale=0.5}{2}{1/2,2/1}{0/1,0/2,1/0} (\pi^c)=
\pattern{scale=0.5}{2}{1/1,2/2}{1/0,2/1,2/2} (\pi^{r\circ c})=
\pattern{scale=0.5}{2}{1/2,2/1}{1/2,2/0,2/1} (\pi^{r})\\
=&\pattern{scale=0.5}{2}{1/1,2/2}{0/0,1/0,2/1}(\pi^{-1})
=\pattern{scale=0.5}{2}{1/2,2/1}{0/1,2/0,1/0} (\tau^r)=
\pattern{scale=0.5}{2}{1/1,2/2}{0/1,1/2,2/2} (\tau^{r\circ c})=
\pattern{scale=0.5}{2}{1/2,2/1}{2/1,0/2,1/2} (\tau^{c})
\end{align*}
with  $\tau=\pi^{-1}$.

\begin{lemma}\label{lemma1:main1}
For $\pi\in S_n$, we have
\begin{align}
\earec(\pi)&=\pattern{scale=0.5}{2}{1/2,2/1}{0/2,1/2,2/0} (\pi)
=\pattern{scale=0.5}{2}{1/2,2/1}{0/1,1/1,2/0} (\pi)
=\pattern{scale=0.5}{2}{1/2,2/1}{0/1,0/2,2/0} (\pi)
=\pattern{scale=0.5}{2}{1/2,2/1}{1/1,1/2,2/0} (\pi)
,\label{eq1:lemma1}\\
\erec(\pi)&=\pattern{scale=0.5}{2}{1/2,2/1}{0/2,1/0,2/0} (\pi)
=\pattern{scale=0.5}{2}{1/2,2/1}{0/2,1/1,2/1} (\pi)
=\pattern{scale=0.5}{2}{1/2,2/1}{0/2,2/0,2/1} (\pi)
=\pattern{scale=0.5}{2}{1/2,2/1}{0/2,1/0,1/1} (\pi).\label{eq2:lemma1}
\end{align}
\end{lemma}
\begin{proof}
We just prove \eqref{eq1:lemma1} as the proof of \eqref{eq2:lemma1} is similar. In the rook placement representation of a permutation
$\pi\in S_n$ 
the rook $y=(i, \pi(i))$ is an exclusive antirecord iff there is a 
another rook $x=(j,\pi(j))$ at left of $y$, i.e., 
 $j<i$ and higher than $x$, i.e., $\pi(j)>\pi(i)$.
 Hence there are four unique choices for such a rook $x$: 
 the \emph{highest, lowest,  farthest} and \emph{nearest}. This corresponds
 to the four mesh patterns in \eqref{eq1:lemma1}, respectively. 
 \end{proof}
 \begin{remark}
As $\earec(\pi)=\erec(\pi^{\r\circ c})$ for $\pi\in S_n$, we can also derive 
\eqref{eq2:lemma1} from \eqref{eq1:lemma1}. 
 \end{remark}

%


\begin{theorem}\label{main1} There exists an involution $\Phi$ on $S_n$ such that for $\pi\in S_n$,
$$(\pattern{scale=0.5}{2}{1/2,2/1}{0/1,2/0,1/0},  \pattern{scale=0.5}{2}{1/2,2/1}{0/1,0/2,2/0,1/0,1/2},\pattern{scale=0.5}{2}{1/2,2/1}{0/1,1/0,2/0,1/2})\pi=
(\pattern{scale=0.5}{2}{1/2,2/1}{0/2,1/0,2/0},  \pattern{scale=0.5}{2}{1/2,2/1}{0/1,0/2,1/0,2/0,1/1}, \pattern{scale=0.5}{2}{1/2,2/1}{0/2,1/0,2/0,1/1})\Phi(\pi).
$$
\end{theorem}

\begin{corollary}\label{cor:main1}
The triple pattern $(\Nr.3, \Nr.48, \Nr.53)$ is equidistributed with 
$(\erec, \Nr.50, \Nr.54)$ on $S_n$.
\end{corollary}
\begin{proof} For any $\pi\in S_n$ we have 
\begin{align}\label{eq:triple}
\begin{bmatrix}
\Nr.3\\
\\
\Nr.48\\
\\
\Nr.53
\end{bmatrix}\pi
=
\begin{bmatrix}
\pattern{scale=0.5}{2}{1/1,2/2}{0/0,0/1,1/2}\\
\\
\pattern{scale=0.5}{2}{1/1,2/2}{ 0/0,0/1,1/2,2/1,2/2}\\
\\
\pattern{scale=0.5}{2}{1/1,2/2}{ 0/0,0/1,1/2,2/1}
\end{bmatrix}\pi
=
\begin{bmatrix}
\pattern{scale=0.5}{2}{1/1,2/2}{0/0,1/0,2/1} \\
\\
\pattern{scale=0.5}{2}{1/1,2/2}{0/0,1/0,1/2,2/1,2/2}\\
\\
\pattern{scale=0.5}{2}{1/1,2/2}{0/0,1/0,2/1,1/2}
\end{bmatrix}\pi^{-1}
=
\begin{bmatrix}
\pattern{scale=0.5}{2}{1/2,2/1}{0/1,2/0,1/0} \\
\\
\pattern{scale=0.5}{2}{1/2,2/1}{ 0/1,0/2,1/0,1/2,2/0}  \\
\\
\pattern{scale=0.5}{2}{1/2,2/1}{ 0/1,1/0,1/2,2/0}
\end{bmatrix}(\pi^{-1})^\r
\end{align}
and 
\begin{align}\label{eq:double}
(\Nr.50, \Nr.54)\pi=
(\pattern{scale=0.5}{2}{1/1,2/2}{ 0/0,0/1,1/1,2/2,1/2},
\pattern{scale=0.5}{2}{1/1,2/2}{ 0/0,0/1,1/1,2/2})\pi
= 
 (\pattern{scale=0.5}{2}{1/1,2/2}{0/0,1/0,1/1,2/1,2/2}, 
 \pattern{scale=0.5}{2}{1/1,2/2}{0/0,1/0,1/1,2/2}) \pi^{-1}
 =
(\pattern{scale=0.5}{2}{1/2,2/1}{ 0/1,0/2,1/0,1/1,2/0},
\pattern{scale=0.5}{2}{1/2,2/1}{ 0/2,1/0,1/1,2/0})(\pi^{-1})^\r.
\end{align}
By Theorem~\ref{main1} the result follows from \eqref{eq2:lemma1}, \eqref{eq:triple} and \eqref{eq:double}.
\end{proof}

\begin{corollary}\label{cor:conj1}
Conjecture~\ref{conj1} holds true.
\end{corollary}
\begin{proof}
By Corollary~\ref{cor:main1} this follows from \eqref{cf-pattern} with $x=v=q=1$. 
\end{proof}
%

As the equidistribution of $\Nr. 48$ and $\Nr. 49$ is known~\cite[Theorem~5.1]{KZ19}, Corollary~\ref{cor:main1} confirms  two conjectured equidistributions in Table~\ref{tab-2}.

\begin{theorem}\label{main2}
 There exist an involution $\Psi$ on $S_n$ such that
 for $\pi\in S_n$,
$$(\pattern{scale=0.5}{2}{1/2,2/1}{1/0,2/0,2/2}, \pattern{scale=0.5}{2}{1/2,2/1}{1/0, 1/1,2/0,2/2},\pattern{scale=0.5}{2}{1/2,2/1}{1/0, 1/1, 1/2,2/0,2/2})(\pi)=(\pattern{scale=0.5}{2}{1/2,2/1}{2/1,1/0,2/0}, \pattern{scale=0.5}{2}{1/2,2/1}{2/1,1/0,2/0,1/2}, \pattern{scale=0.5}{2}{1/2,2/1}{2/1,1/0,2/0,1/1,1/2})
\Psi(\pi)).
$$
\end{theorem}

For the patterns $\Nr.23$ and $\Nr.24$, we have
\begin{align*}
(\Nr.23, \Nr.24)\pi=
(\pattern{scale=0.5}{2}{1/1,2/2}{1/0, 1/1, 1/2,0/0,0/2},
\pattern{scale=0.5}{2}{1/1,2/2}{0/1,1/0,0/0,1/1,1/2})
\pi=(\pattern{scale=0.5}{2}{1/2,2/1}{1/0, 1/1, 1/2,2/0,2/2},
\pattern{scale=0.5}{2}{1/2,2/1}{2/1,1/0,2/0,1/1,1/2})
\pi^r.
\end{align*}
By Theorem~\ref{main2}, we confirm another conjecture in Table~\ref{tab-2}, i.e., the patterns $\Nr.23$ and $\Nr.24$ are equidistributed.

%

We shall prove Theorem~\ref{main1} and Theorem~\ref{main2}
in Section~2 and Section~3, respectively, and make a connection  between
  pattern Nr. 14 and the statistic \emph{succession}  in permutations in Section~4.

\section{Proof of Theorem~\ref{main1}}

For $\pi\in S_n$ let $\AREC(\pi)=(i_1, i_2, \ldots, i_l)$ be the sequence of antirecord positions of $\pi$ from left to right. So 
$\pi(i_1)=1$, $i_1<\cdots <i_l$ and $i_l=n$.
For each antirecord position $i_k$ 
 define two mappings
 \begin{subequations}\label{eq:defanti1}
\begin{align}
&\varphi_1^{(i_k)}: \pi\mapsto \pi'\\ 
&\varphi_2^{(i_k)}: \pi\mapsto \pi''
\end{align}
\end{subequations}
as follows: 
\begin{itemize}
\item let  $w=w_1\ldots w_r$  be  the subword of $\pi$ 
consisting of letters  greater than $\pi(i_k)$ on the left of $\pi(i_k)$ (resp. $\pi(i_{k-1})$);
\item let 
$w'=w_1'\ldots w_r'$   be the word obtained by substituting 
the $j$th largest letter  with  the $j$th smallest letter in $w$
for $j=1,\ldots, r$;
\item  let $\pi'$ (resp. $\pi''$) be the word obtained by 
replacing $w_j$  with  $w_j'$ for $j=1,\ldots, r$ 
 in $\pi$. 
\end{itemize}

\begin{remark}
By convention, we define $\varphi_2^{(i_1)}$ to be the identity mapping.
Clearly  the two operations keep the sequence of 
 antirecords for both values and positions, that is,
 \begin{subequations}
 \begin{align}
 \AREC(\pi)&=\AREC(\pi')=\AREC(\pi'')\\
\pi'(i_k)&=\pi''(i_k)=\pi(i_k)\quad \text{for} \quad k=1, \ldots l.
\end{align}
\end{subequations}
\end{remark}

Let $P=\{p_1<\cdots <p_r\}$ and $Q=\{q_1<\cdots <q_r\}$ be two ordered sets and 
 $\pi=p_{1}\ldots p_{r}$ and $\tau=q_{1}\ldots q_{r}$ are permutations of $P$ and $Q$, respectively.  We say that $\pi$ and $\tau$ are  \emph{order isomorphic} and write $\pi\sim \tau$ 
   if for any two  indices $r$ and $s$ we have the equivalence 
 $p_r<p_s$ if and only if $q_r<q_s$. In other words, $\tau$ is the
  permutation obtained from $\pi$ by substituting $p_i$ by $q_i$ for $i=1, \ldots, r$.

Let $w=w_1\ldots w_n$ be a permutation of $a_1<a_2<\cdots<a_n$. We define
the \emph{complement} of $w$ by $w^c$\footnote[1]{When $a_i=i$, $w^c$ reduces to $\pi^c$, see \eqref{def:comlpi}.}, which is the word obtained by substituting $a_i$ by $a_{n+1-i}$ in $w$ for $i=1, \ldots, n$. 
If $x$ is a subset of  the letters in $w$, we write $[w]_x$ as the subword of $w$ consisting of the letters $a\in x$. 
\begin{lemma}\label{property:compl}
\begin{enumerate}
\item If $w=w_1w_2$ and $w^c=w_1'w_2'$, then $(w_1')^c\sim w_1$.\footnote[2]{The word $w_1'$ is the complement of $w_1$ in the alphabet of $w$, while $(w_1')^c$ is the complement of $w_1'$ in the alphabet of $w_1'$.}
\item  Let  $w=w_1w_2w_3$ and  $v=v_1v_2v_3$ with $|w_1|=|v_1|$. If $w_1w_2\sim v_1v_2$ with
$(w_1w_2)^c=w_1'w_2'$ and $(v_1v_2)^c=v_1'v_2'$,  then $w_1\sim v_1$, $w_2\sim v_2$, $w_1'\sim v_1'$ and $w_2'\sim v_2'$. 
Moreover,
we have    
$(w'_1)^c\sim (v'_1)^c$ and $(w'_2)^c\sim (v'_2)^c$.
\item  If $w\sim v$ and $[w]_{x} =[v]_{x}$ for some set $x$ of some common 
letters in $w$ and $v$, then 
\begin{itemize}
\item $w^c\sim v^c$ and $[w^c]_{x}=[v^c]_{x}$.
\item $[w]_{y}\sim [v]_{z}$, where $y$ (resp. $z$) is the complementary of $x$ in the alphabet of $w$ (resp. $v$).
\end{itemize}
\end{enumerate}
\end{lemma}
\begin{proof}
The verification is easy and left to the reader.
\end{proof}

For example, 
if $w=3\,5\,9\,1\,4\,7\,2\,8\,6$, then $w^c=7\,5\,1\,9\,6\,3\,8\,2\,4$. Let $w=w_1w_2$ with 
$w_1=3\,5\,9\,1\,4\,7$ and $w_2=2\,8\,6$, then $w_1'=7\,5\,1\,9\,6\,3$ and $(w_1')^c=3\,6\,9\,1\,5\,7$.
Clearly $(w_1')^c\sim w_1$ and $[(w_1')^c]_x=[w_1]_x$ with
$x=\{1,\,3, \,9\}$.  We see that $w_1^c=7\,4\,1\,9\,5\,3$ and $[w_1']_x=[w_1^c]_x=1\,9\,3$.

\begin{lemma}\label{lem:opera1}
For any antirecord position $i$ of $\pi\in S_n$ 
the mappings $\varphi_1^{(i)}$  and  $\varphi_2^{(i)}$ are involutions and 
commute, namely,
\begin{equation}\label{eq1:lem2.5}
\varphi_1^{(i)}\circ \varphi_1^{(i)}(\pi)=
\varphi_2^{(i)}\circ \varphi_2^{(i)}(\pi)=\pi
\end{equation}
and 
\begin{equation}\label{eq2:lem2.5}
\varphi_2^{(i)}\circ \varphi_1^{(i)}(\pi)=\varphi_1^{(i)}\circ \varphi_2^{(i)}(\pi).
\end{equation}
\end{lemma}
\begin{proof}
From the definitions of $\varphi_1^{(i)}$ and $\varphi_2^{(i)}$ in Eq.~\eqref{eq:defanti1}, it is easy to check Eq.~\eqref{eq1:lem2.5} holds and 
$$
\varphi_2^{(i)}\circ \varphi_1^{(i)}(\pi) \sim \varphi_1^{(i)}\circ \varphi_2^{(i)}(\pi).
$$
Since the set of letters greater than $\pi(i)$ on the left of $\pi(i)$ are invariant under the operation $\varphi_1^{(i)}$ and  $\varphi_2^{(i)}$ on $\pi$, we obtain Eq.\eqref{eq2:lem2.5} immediately. 
\end{proof}

 Let  $\pi\in S_n$ with  sequence of antirecord positions 
$\AREC(\pi)=(i_1,i_2,\ldots, i_l).$ 
 We define the operation  $\Phi$ on $\pi$ by 
 \begin{align}\label{eq:defPhi}
 \Phi(\pi)=\varphi^{(i_1)}\circ\varphi^{(i_2)}\circ\cdots\circ
 \varphi^{(i_l)}(\pi)
 \end{align}
 with $\varphi^{(i_k)}=\varphi_2^{(i_k)}\circ \varphi_1^{(i_k)}$
for $k=1, \ldots, l$.

\begin{figure}[t]
	\begin{center}
		\begin{tikzpicture}[scale=1.2, baseline=(current bounding box.center)]
		\draw (0.01,0.01) grid (3+0.99,3+0.99);
		\filldraw (3,3) circle (3pt) node[below left] {$i_k$};		
		\filldraw (2,2) circle (3pt) node[below left] {$i_{k-1}$};
		\filldraw (1,1) circle (3pt) node[below left] {$i_{k-2}$};
		\draw (0.5,2.5) node{$u_2$};
		\draw (0.5,3.5) node{$u_1$};	
		\draw (1.5,2.5) node{$v_2$};
		\draw (1.5,3.5) node{$v_1$};
                 \draw (2.5,3.5) node{$w$};	
\end{tikzpicture}
		\end{center}
		\medskip
	\caption{The decomposition of three consecutive  anti-records of $\pi$}\label{Fig:1}
\end{figure}
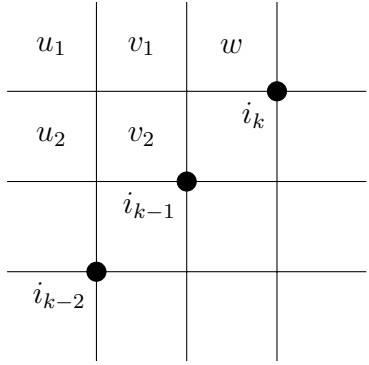

\begin{lemma}\label{lemma2.6:main1} 
For $\pi\in S_n$ with $\AREC(\pi)=\{i_1, \ldots, i_l\}$. The mappings  
$g:=\varphi^{(i_{k-1})}$ and $f:=\varphi^{(i_k)}$ commute, i.e., 
$$
g\circ f(\pi)=f\circ g(\pi).
$$
\end{lemma}
\begin{proof}
 We write the permutation  $\pi=\pi(1)\ldots \pi(n)$ as 
$\pi=u\pi(i_{k-2})v\pi(i_{k-1})w\pi(i_{k})x$ for $3\leq k\leq l$, and 
\begin{itemize}
\item[(i)] $u_1$ (resp. $v_1$,  $w$) as  the subword consisting of letters greater than $\pi(i_k)$ in $u$ (resp. $v$, $w$);
\item[(ii)] $u_2$ (resp. $v_2$) as  the subword consisting of letters between 
 $\pi(i_{k-1})$  and $\pi(i_k)$  in $u$ (resp. $v$); 
\end{itemize}
 see    Figure~\ref{Fig:1}.
For convenience, we introduce the following notations:
\begin{subequations}
\begin{align}
f(\pi):&=u'\pi(i_{k-2})v'\pi(i_{k-1})w'\pi(i_{k})x\label{eq:f}\\
g\circ f (\pi):&=\tilde u\pi(i_{k-2})\tilde v\pi(i_{k-1})w'\pi(i_{k})x,\label{eq:gf}\\
g(\pi):&=u''\pi(i_{k-2})v''\pi(i_{k-1})w\pi(i_{k})x,\label{eq:g}\\
f\circ g (\pi):&=\hat u\pi(i_{k-2})\hat v\pi(i_{k-1})w'\pi(i_{k})x.\label{eq:fg}
\end{align}
\end{subequations}
We will use the similar notations $u_i', \hat{u}_i, \tilde{u_i}, u_i'', v_i', \hat{v}_i, \tilde{v_i}, v_i''$ as in (i) and (ii) for $i=1, 2$.

By the definition of  $f$ and $g$ and Lemma~\ref{property:compl}, we have the following facts,
\begin{enumerate}
\item Under the operation $f$ (cf. \eqref{eq:f}), as $u_1v_1\sim u_1'v_1'$, $u_2'v_2'=u_2v_2$, we have $u'v'\sim uv$;
\item Applying $g$ to  $f(\pi)$ and $\pi$ (cf. \eqref{eq:gf} and \eqref{eq:g}), respectively,   we see that $\tilde u\tilde v\sim u''v''$  by (1) and Lemma~\ref{property:compl}, hence
$\tilde u_2\tilde v_2=u''_2v''_2$;
\item Applying $f$ to $g(\pi)$ (cf. \eqref{eq:g} and \eqref{eq:fg}),  we have  $u_2''v_2''=\hat u_2\hat v_2$, 
 combining  with  (2) yields 
  $\hat u_2\hat v_2=\tilde u_2\tilde v_2$, that is $\hat u_2=\tilde u_2$ and 
  $\hat v_2=\tilde v_2$;
 \item Applying $f$ to $g(\pi)$ (cf. \eqref{eq:g} and \eqref{eq:fg}),  using Lemma~\ref{property:compl}
 we have   $\hat u\hat v\sim u''v''$, from $u''v''\sim \tilde u\tilde v$ (cf. (2)) we derive 
 $\hat u\hat v\sim \tilde u\tilde v$.  Combining with Lemma~\ref{property:compl} and    (3) yields
 $\hat u_1\hat v_1\sim \tilde u_1\tilde v_1$.  
 \item 
 Let $\{w\}$ denote the set of letters in $w$.  
 \begin{itemize}
\item                        Applying $f$ to $\pi$ we have 
 $\{u_1'v_1'w'\}=\{ u_1  v_1 w\}$,
 \item applying $g$ to $f(\pi)$ we have 
 $\{u_1'v_1'w'\}=\{\tilde u_1\tilde v_1w'\}$, 
  \item applying $g$ to $\pi$ we have 
 $\{u_1''v_1''w\}=\{u_1 v_1 w\}$,
 \item applying $f$ to $g(\pi)$ we have 
 $\{u_1''v_1''w\}=\{\hat u_1\hat  v_1 w'\}$.
\end{itemize}
 Thus  $\{\tilde u_1\tilde v_1\}=\{\hat u_1\hat  v_1\}$. It follows from (4) that 
 $\hat u_1= \tilde u_1$ (resp. $\hat v_1=\tilde v_1$).
 \end{enumerate}
 Summarizing the above facts we have proved $f\circ g=g\circ f$.
 \end{proof}

\begin{lemma}\label{lemma2.7:main1} The mapping  $\varphi^{(i_k)}$ is an involution   such that for $\pi\in S_n$ and $r\neq k$,
\begin{subequations}
\begin{align}\label{eq:phia}
(\pattern{scale=0.5}{2}{1/2,2/1}{0/2,1/0,2/0}, \pattern{scale=0.5}{2}{1/2,2/1}{0/2,1/0,2/0,1/1}, 
\pattern{scale=0.5}{2}{1/2,2/1}{0/1,0/2,1/0,2/0,1/1})_k\;\pi
=(\pattern{scale=0.5}{2}{1/2,2/1}{0/1,2/0,1/0}, \pattern{scale=0.5}{2}{1/2,2/1}{0/1,1/0,2/0,1/2}, \pattern{scale=0.5}{2}{1/2,2/1}{0/1,0/2,2/0,1/0,1/2})_k\;\varphi^{(i_k)}(\pi),\\
(\pattern{scale=0.5}{2}{1/2,2/1}{0/1,2/0,1/0}, \pattern{scale=0.5}{2}{1/2,2/1}{0/1,1/0,2/0,1/2}, \pattern{scale=0.5}{2}{1/2,2/1}{0/1,0/2,2/0,1/0,1/2})_r\;\pi
=
(\pattern{scale=0.5}{2}{1/2,2/1}{0/1,2/0,1/0}, \pattern{scale=0.5}{2}{1/2,2/1}{0/1,1/0,2/0,1/2}, \pattern{scale=0.5}{2}{1/2,2/1}{0/1,0/2,2/0,1/0,1/2})_r\;\varphi^{(i_k)}(\pi),\label{eq:phib}\\
(\pattern{scale=0.5}{2}{1/2,2/1}{0/2,1/0,2/0}, \pattern{scale=0.5}{2}{1/2,2/1}{0/2,1/0,2/0,1/1}, 
\pattern{scale=0.5}{2}{1/2,2/1}{0/1,0/2,1/0,2/0,1/1})_r\;\pi
=
(\pattern{scale=0.5}{2}{1/2,2/1}{0/2,1/0,2/0}, \pattern{scale=0.5}{2}{1/2,2/1}{0/2,1/0,2/0,1/1}, 
\pattern{scale=0.5}{2}{1/2,2/1}{0/1,0/2,1/0,2/0,1/1})_r\;\varphi^{(i_k)}(\pi),
\label{eq:phic}
\end{align}
\end{subequations}
where $(\textrm{pattern})_k$ means the number of the patterns between $\pi(i_{k-1})$ and $\pi(i_k)$.
\end{lemma}
\begin{proof}
If  the pair
$(\pi(j),\pi(i_k))$ with $j<i_k$ contributes the pattern $\pattern{scale=0.5}{2}{1/2,2/1}{0/2,1/0,2/0}$(resp. $\pattern{scale=0.5}{2}{1/2,2/1}{0/2,1/0,2/0,1/1}, \pattern{scale=0.5}{2}{1/2,2/1}{0/1,0/2,1/0,2/0,1/1})$,
then  $j>i_{k-1}$ because $\pi(i_{k-1})<\pi(i_{k})$ and 
$\pi(i)>\pi(i_k)$ for $j\leq i<i_k$.
Also, for  $i<j$, we have the equivalence 
 $$
   \pi(i)<\pi(j)\Longleftrightarrow \pi(i_k)<\varphi_1^{(i_k)}(\pi(i))>\varphi_1^{(i_k)}(\pi(j)),
   $$
as $\varphi_2^{(i_k)}$ will affect only the letters at the left of 
$\pi(i_{k-1})$. Thus  we have proved \eqref{eq:phia}.

Next, recall that the operation $\varphi^{(i_k)}$ keeps the 
sequence of antirecords for both positions and values. The two identities \eqref{eq:phib} and \eqref{eq:phic} are clear if $r>k$.
Assume that $r<k$ and $[\pi]_r=\pi(1)\ldots \pi(i_{k-1})$. 
By Lemma~\ref{property:compl}, 
after two operations
$\varphi_1^{(i_k)}$ and $\varphi_2^{(i_k)}$  the permutation
$\varphi^{(i_k)}([\pi]_r)$ is  isomorphic with $[\pi]_r$. This proves \eqref{eq:phib} and \eqref{eq:phic}.

\end{proof}

\begin{proof}[Proof of Theorem~\ref{main1}]
By \eqref{eq:defPhi}  the reverse of the mapping $\Phi$ is given by 
\begin{align}\label{eq:defPhireverse}
 \Phi^{-1}(\pi)=\varphi^{(i_l)}\circ\cdots\circ \varphi^{(i_2)}\circ
 \varphi^{(i_1)}(\pi).
 \end{align}
Theorem~\ref{main1} follows from Lemma~\ref{lem:opera1},  Lemma~\ref{lemma2.6:main1} and Lemma~\ref{lemma2.7:main1}.
\end{proof}

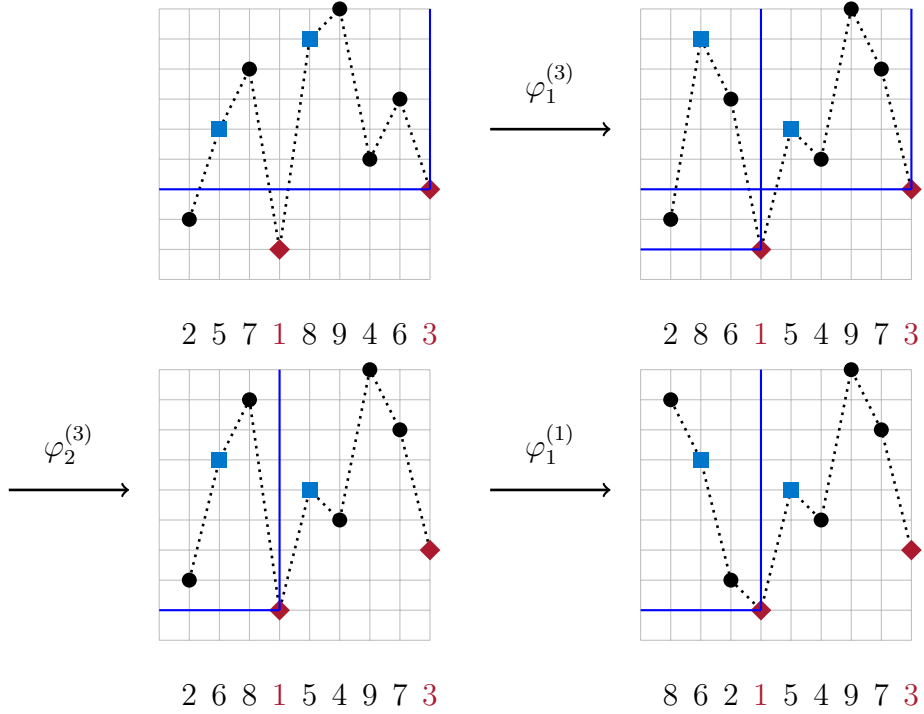
\begin{figure}
\begin{center}
\begin{tikzpicture}[scale=0.4] 	
\draw[step=1,lightgray,thin] (0,0) grid (9,9); 
	\tikzstyle{ridge}=[draw, line width=1, dotted, color=black] 
	\path[ridge] (1,2)--(2,5)--(3,7)--(4,1)--(5,8)--(6,9)--(7,4)--(8,6)--(9,3); 
	\tikzstyle{node0}=[circle, inner sep=2, fill=black] 
	\tikzstyle{node1}=[rectangle, inner sep=3, fill=mhcblue] 
	\tikzstyle{node2}=[diamond, inner sep=2, fill=davidsonred] 
	\node[node0] at (1,2) {}; 
	\node[node1] at (2,5) {}; 
	\node[node0] at (3,7) {}; 
	\node[node2] at (4,1) {}; 
	\node[node1] at (5,8) {}; 
	\node[node0] at (6,9) {}; 
	\node[node0] at (7,4) {}; 
	\node[node0] at (8,6) {}; 
	\node[node2] at (9,3) {}; 
	\draw[thick,blue,-] (0,3) to (9,3);	
	\draw[thick,blue,-] (9,3) to (9,9);		
	\tikzstyle{pi}=[above=-1] 
	\node[pi] at (1,0) {$2$}; 
	\node[pi] at (2,0) {$5$}; 
	\node[pi] at (3,0) {$7$}; 
         \node[pi, color=davidsonred] at (4,0) {$1$}; 
         \node[pi] at (5,0) {$8$};
	\node[pi] at (6,0) {$9$}; 
	\node[pi] at (7,0) {$4$}; 
	\node[pi] at (8,0) {$6$}; 	
         \node[pi, color=davidsonred] at (9,0) {$3$}; 

	\path[draw,line width=1,->] (11,5)--(15,5); 
	\node[pi] at (13,8){$\varphi_1^{(3)}$};
	\begin{scope}[shift={(16,0)}] 
	\draw[step=1,lightgray,thin] (0,0) grid (9,9); 
	\path[ridge] (1,2)--(2,8)--(3,6)--(4,1)--(5,5)--(6,4)--(7,9)--(8,7)--(9,3); 
	\node[node0] at (1,2) {}; 
	\node[node1] at (2,8) {}; 
	\node[node0] at (3,6) {}; 
	\node[node2] at (4,1) {}; 
	\node[node1] at (5,5) {}; 
	\node[node0] at (6,4) {}; 
	\node[node0] at (7,9) {}; 
	\node[node0] at (8,7) {};
	\node[node2] at (9,3) {};
       	\draw[thick,blue,-] (0,3) to (9,3);	
	\draw[thick,blue,-] (9,3) to (9,9);	
		\draw[thick,blue,-] (0,1) to (4,1);	
	\draw[thick,blue,-] (4,1) to (4,9);				       
       
	\node[pi] at (1,0) {$2$}; 
	\node[pi] at (2,0) {$8$}; 
	\node[pi] at (3,0) {$6$}; 
      \node[pi, color=davidsonred] at (4,0) {$1$}; 
	\node[pi] at (5,0) {$5$}; 
	\node[pi] at (6,0) {$4$};
	\node[pi] at (7,0) {$9$}; 
	\node[pi] at (8,0) {$7$};
      \node[pi, color=davidsonred] at (9,0) {$3$}; 
	\end{scope}
	
		\path[draw,line width=1,->] (-5,-7)--(-1,-7); 
	\node[pi] at (-3,-4){$\varphi_2^{(3)}$};
	\begin{scope}[shift={(0,-12)}] 
	\draw[step=1,lightgray,thin] (0,0) grid (9,9); 
	\path[ridge] (1,2)--(2,6)--(3,8)--(4,1)--(5,5)--(6,4)--(7,9)--(8,7)--(9,3); 
	\node[node0] at (1,2) {}; 
	\node[node1] at (2,6) {}; 
	\node[node0] at (3,8) {}; 
	\node[node2] at (4,1) {}; 
	\node[node1] at (5,5) {}; 
	\node[node0] at (6,4) {}; 
	\node[node0] at (7,9) {}; 
	\node[node0] at (8,7) {};
	\node[node2] at (9,3) {};
		\draw[thick,blue,-] (0,1) to (4,1);	
	\draw[thick,blue,-] (4,1) to (4,9);			
	
	\node[pi] at (1,0) {$2$}; 
	\node[pi] at (2,0) {$6$}; 
	\node[pi] at (3,0) {$8$}; 
      \node[pi, color=davidsonred] at (4,0) {$1$}; 
	\node[pi] at (5,0) {$5$}; 
	\node[pi] at (6,0) {$4$};
	\node[pi] at (7,0) {$9$}; 
	\node[pi] at (8,0) {$7$};
      \node[pi, color=davidsonred] at (9,0) {$3$}; 
	\end{scope}

			\path[draw,line width=1,->] (11,-7)--(15,-7); 
	\node[pi] at (13,-4){$\varphi_1^{(1)}$};
	\begin{scope}[shift={(16,-12)}] 
	\draw[step=1,lightgray,thin] (0,0) grid (9,9); 
	\path[ridge] (1,8)--(2,6)--(3,2)--(4,1)--(5,5)--(6,4)--(7,9)--(8,7)--(9,3); 
	\node[node0] at (1,8) {}; 
	\node[node1] at (2,6) {}; 
	\node[node0] at (3,2) {}; 
	\node[node2] at (4,1) {}; 
	\node[node1] at (5,5) {}; 
	\node[node0] at (6,4) {}; 
	\node[node0] at (7,9) {}; 
	\node[node0] at (8,7) {};
	\node[node2] at (9,3) {};
			\draw[thick,blue,-] (0,1) to (4,1);	
	\draw[thick,blue,-] (4,1) to (4,9);

	\node[pi] at (1,0) {$8$}; 
	\node[pi] at (2,0) {$6$}; 
	\node[pi] at (3,0) {$2$}; 
      \node[pi, color=davidsonred] at (4,0) {$1$}; 
	\node[pi] at (5,0) {$5$}; 
	\node[pi] at (6,0) {$4$};
	\node[pi] at (7,0) {$9$}; 
	\node[pi] at (8,0) {$7$};
      \node[pi, color=davidsonred] at (9,0) {$3$}; 
	\end{scope}

\end{tikzpicture}
\end{center}

\caption{The involution $\Phi$ on the permutation $257189463$}\label{fig:2}
\end{figure}

\begin{example}
We show the process of the involution $\Phi$ in Figure~\ref{fig:2}, 
For $\pi= 2\,5\,7\,1\,8\,9\,4\,6\,3$, we have $\AREC(\pi)=(4, 9)$.
 We proceed from right to left. 
 
\begin{enumerate} 
\item For position $9$ with value $3$, w have 
$w=5\, 7\, 8\, 9\, 4\, 6$ and $w'=8\,6\,5\,4\,9\,7$.
Thus  $\varphi_1^{(9)}: \pi\mapsto \pi'=2\,8\,6\,1\,5\,4\,9\,7\,3$. Next,  we have 
$w=8\,6$ and $w'=6\,8$. 
Thus $\varphi_2^{(9)}: \pi'\mapsto \pi''=2\,6\,8\,1\,5\,4\,9\,7\,3$.
\item 
For position $4$ with value $1$ we have 
$w=2\,6\,8$ and $w'=8\,6\,2$. Finally we obtain
$\Phi(\pi)=8\,6\,2\,1\,5\,4\,9\,7\,3$.
\end{enumerate}

Now, we check the mesh patterns.
\begin{itemize}
\item
First, $\varphi_1^{(9)}: \pi= 2\,5\,7\,1\,8\,9\,4\,6\,3 \mapsto \pi'=2\,8\,6\,1\,5\,4\,9\,7\,3$, 
 the pair $(8,3)$ of $\pi$ contributes the pattern $\pattern{scale=0.5}{2}{1/2,2/1}{0/2,1/0,2/0}$ without the patterns $\pattern{scale=0.5}{2}{1/2,2/1}{0/2,1/0,2/0,1/1}, \pattern{scale=0.5}{2}{1/2,2/1}{0/1,0/2,1/0,2/0,1/1}$), 
the pair $(5,3)$ of $\pi'$ contributes the pattern $\pattern{scale=0.5}{2}{1/2,2/1}{0/1,2/0,1/0}$ without the patterns $\pattern{scale=0.5}{2}{1/2,2/1}{0/1,1/0,2/0,1/2}, \pattern{scale=0.5}{2}{1/2,2/1}{0/1,0/2,2/0,1/0,1/2}$,
 the operations $\varphi_2^{(9)}, \varphi_1^{(1)}$ do not change the corresponding mesh pattens at position $9$ of $\pi'$.

\item
Second, $\varphi_2^{(9)}\circ \varphi_1^{(9)}:\pi=2\,5\,7\,1\,8\,9\,4\,6\,3 \mapsto \pi''=2\,6\,8\,1\,5\,4\,9\,7\,3$  it is easy to see that $257\sim 268$. The pair $(5, 1)$ of $\pi$ contributes the patterns $\pattern{scale=0.5}{2}{1/2,2/1}{0/2,1/0,2/0}, \pattern{scale=0.5}{2}{1/2,2/1}{0/2,1/0,2/0,1/1}$ without the pattern $\pattern{scale=0.5}{2}{1/2,2/1}{0/1,0/2,1/0,2/0,1/1}$,  
the pair $(6, 1)$ of $\pi''$ also contributes the pattern $\pattern{scale=0.5}{2}{1/2,2/1}{0/2,1/0,2/0}$,$\pattern{scale=0.5}{2}{1/2,2/1}{0/2,1/0,2/0,1/1}$ without the pattern $\pattern{scale=0.5}{2}{1/2,2/1}{0/1,0/2,1/0,2/0,1/1}$,  
$\varphi_1^{(1)}:\pi''=2\,6\,8\,1\,5\,4\,9\,7\,3 \mapsto \pi'''=8\,6\,2\,1\,5\,4\,9\,7\,3$, the pair $(6,1)$ of $\pi'''$ contributes the pattern $\pattern{scale=0.5}{2}{1/2,2/1}{0/1,2/0,1/0}$, $\pattern{scale=0.5}{2}{1/2,2/1}{0/1,1/0,2/0,1/2}$ without the pattern $\pattern{scale=0.5}{2}{1/2,2/1}{0/1,0/2,2/0,1/0,1/2}$.
\end{itemize}

\end{example}

\section{Proof of Theorem~\ref{main2}}
First we introduce two mappings different from Section~2. For $\pi\in S_n$, 
recall that $\AREC(\pi)=(i_1,i_2,\ldots, i_l)$ be the sequence of antirecord positions  of $\pi$ from left to right. 
For any antirecord position $i_k$ we define two mappings
\begin{subequations}\label{eq:defanti2}
\begin{align}
&\psi_1^{(i_k)}: \pi\mapsto \pi'\\
&\psi_2^{(i_k)}: \pi\mapsto \pi''
\end{align}
\end{subequations}
as follows: 
\begin{itemize}
\item let  $w=w_1\ldots w_r$ is  the subword of $\pi$ 
consisting of letters  greater than $\pi(i_k)$ on the right side of $\pi(i_{k-1})$ (resp. $\pi(i_{k})$)   with $\pi(i_0)=0$;
\item let
$w'=w_1'\ldots w_r'$  be the word obtained by substituting 
the $j$th largest letter  with  the $j$th smallest letter in $w$ for $j=1,\ldots, r$;
\item let  $\pi'$ (resp. $\pi''$) is defined to be the word obtained by 
replacing $w_j$ with  $w_j'$ in $\pi$. 
\end{itemize}
 Note that  $\pi'(i_k)=\pi(i_k)$.

\begin{lemma}\label{lem:opera2}
For any antirecord positions $i_{k-1}$ and $i_k$ of $\pi\in S_n$ 
the mappings $\psi_1^{(i_k)}$  and  $\psi_2^{(i_k)}$ are involutions and 
commutate, namely,
\begin{equation}\label{eq1:lem3.1}
\psi_1^{(i_k)}\circ \psi_1^{(i_k)}(\pi)=
\psi_2^{(i_k)}\circ \psi_2^{(i_k)}(\pi)=\pi
\end{equation}
and
\begin{equation}\label{eq2:lem3.1}
\psi_2^{(i_k)}\circ \psi_1^{(i_k)}(\pi)=\psi_1^{(i_k)}\circ \psi_2^{(i_k)}(\pi).
\end{equation}
Let  $\psi^{(i_k)}=\psi_2^{(i_k)}\circ \psi_1^{(i_k)}$.  Then $\psi^{(i_k)}(\pi)$ and $\pi$ have the same  sequence of antirecord 
positions.
\end{lemma}
 
\begin{proof}
From the definitions of $\psi_1^{(i_k)}$ and $\psi_2^{(i_k)}$ in Eq.~\eqref{eq:defanti2}, it is easy to check Eq.~\eqref{eq1:lem3.1} holds and 
$$
\psi_2^{(i_k)}\circ \psi_1^{(i_k)}(\pi) \sim \psi_1^{(i_k)}\circ \psi_2^{(i_k)}(\pi).
$$
Since the set of letters greater than $\pi(i_k)$ on the right of $\pi(i_{k-1})$ are invariant under the operation $\psi_1^{(i_k)}$ and  $\psi_2^{(i_k)}$ on $\pi$, we obtain Eq.~\eqref{eq2:lem3.1} immediately. 
\end{proof}

%

\begin{lemma}\label{lemma2:main2} 
For $\pi\in S_n$ with $\AREC(\pi)=\{i_1, \ldots, i_l\}$. For $k=2, \ldots, l$ the mappings  
$\psi^{(i_{k-1})}$ and $\psi^{(i_k)}$ commute, i.e., 
$$
\psi^{(i_k)}\circ \psi^{(i_{k-1})}(\pi)=\psi^{(i_{k-1})}\circ \psi^{(i_k)}(\pi).
$$
\end{lemma}
\begin{proof}
For  the permutation $\pi=\pi(1)\ldots \pi(n)$ we write
$\pi=u\pi(i_{k-1})v\pi(i_{k})w$ and 
\begin{subequations}
\begin{align}
\psi^{(i_{k-1})}(\pi):&=\pi'=u'\pi(i_{k-1})v'\pi'(i_{k})w'\label{eq:barf}\\
\psi^{(i_k)}\circ \psi^{(i_{k-1})} (\pi):&=\tilde \pi= u'\pi(i_{k-1})\tilde v\pi'(i_{k})\tilde w,\label{eq:bargf}\\
\psi^{(i_k)}(\pi):&=\pi''=u\pi(i_{k-1})v''\pi(i_{k})w'',\label{eq:barg}\\
\psi^{(i_{k-1})}\circ \psi^{(i_k)} (\pi):&=\hat \pi= u'\pi(i_{k-1})\hat v \hat \pi(i_{k})\hat w.\label{eq:barfg}
\end{align}
\end{subequations}
By the definition of  $\psi^{(i_{k})}$, we have the following facts.
\begin{enumerate}
\item Applying $\psi^{(i_{k-1})}$ to $\pi$ (cf. \eqref{eq:barf}) we have   $v'\pi'(i_k)w'\sim v\pi(i_k)w$;
\item Applying $\psi^{(i_k)}$ to $\psi^{(i_{k-1})}(\pi)$ and $\pi$ (cf.\eqref{eq:bargf} and \eqref{eq:barg}), respectively,  we take complement of $v'$ and $v$ once, while twice for
$w'$ and $w$. By Lemma~\ref{property:compl} and (1) we get
$\tilde v\pi'(i_k)\tilde w \sim v''\pi(i_k)w''$;
 \item Applying $\psi^{(i_{k-1})}$ to $\psi^{(i_{k})}(\pi)$ (cf. \eqref{eq:barg} and \eqref{eq:barfg}),  by (1) we have   $\hat v\hat\pi(i_k)\hat w\sim v''\pi(i_k)w''$,  combining with 
 by (2) yields 
 $\hat v\hat \pi(i_k)\hat w\sim \tilde v\pi'(i_k)\tilde w$.
  \item 
 Let $\{w\}$ denote the set of letters in $w$.  
 \begin{itemize}
 \item Applying $\psi^{(i_{k})}$ to $\psi^{(i_{k-1})}(\pi)$ we have 
 $\{u'v'\pi'(i_k)w'\}=\{u'\tilde v\pi'(i_k)\tilde w\}$, 
 \item    applying $\psi^{(i_{k-1})}$ to $\pi$ we have 
 $\{u'v'\pi'(i_{k})w'\}=\{ uv\pi(i_{k})w\}$,
  \item applying $\psi^{(i_{k})}$ to $\pi$ we have 
 $\{uv''\pi(i_{k})w''\}=\{ uv \pi(i_{k})w\}$,
 \item applying $\psi^{(i_{k-1})}$ to $\psi^{(i_{k})}(\pi)$ we have 
 $\{u'\hat v\hat \pi(i_{k})\hat w\}=\{u v''\pi(i_{k})w''\}$,
 \item it follows that $ \{u'\hat v\hat \pi(i_k)\hat w\}=\{u'\tilde v\pi'(i_k)\tilde w\}$ and  $\{\hat v \hat \pi(i_k)\hat w\}=\{\tilde v \pi'(i_k) \tilde w\}$.
\end{itemize}
 
\item  It follows from (3) and (4)  that $\hat v \hat \pi(i_k)\hat w=\tilde v \pi'(i_k) \tilde w$, that is, 
 Thus 
 $\hat v= \tilde v$, $\hat \pi(i_k)=\pi'(i_k)$ and $\hat w=\tilde w$.
 \end{enumerate}
 \end{proof}

\begin{lemma}\label{lem:opera2bis} The mapping $\psi^{(i)}$ is an involution  such that for  $\pi\in S_n$ and $r\neq k$
\begin{subequations}
\begin{align}\label{eq:psia}
(\pattern{scale=0.5}{2}{1/2,2/1}{1/0,2/0,2/2}, \pattern{scale=0.5}{2}{1/2,2/1}{1/0, 1/1,2/0,2/2},\pattern{scale=0.5}{2}{1/2,2/1}{1/0, 1/1, 1/2,2/0,2/2})_k\;\pi
=
(\pattern{scale=0.5}{2}{1/2,2/1}{2/1,1/0,2/0}, \pattern{scale=0.5}{2}{1/2,2/1}{2/1,1/0,2/0,1/2}, \pattern{scale=0.5}{2}{1/2,2/1}{2/1,1/0,2/0,1/1,1/2})_k\;\psi^{(i_k)}(\pi),\\
(\pattern{scale=0.5}{2}{1/2,2/1}{1/0,2/0,2/2}, \pattern{scale=0.5}{2}{1/2,2/1}{1/0, 1/1,2/0,2/2},\pattern{scale=0.5}{2}{1/2,2/1}{1/0, 1/1, 1/2,2/0,2/2})_r\;\pi
=
(\pattern{scale=0.5}{2}{1/2,2/1}{1/0,2/0,2/2}, \pattern{scale=0.5}{2}{1/2,2/1}{1/0, 1/1,2/0,2/2},\pattern{scale=0.5}{2}{1/2,2/1}{1/0, 1/1, 1/2,2/0,2/2})_r\;\psi^{(i_k)}(\pi),\label{eq:psib}\\
(\pattern{scale=0.5}{2}{1/2,2/1}{2/1,1/0,2/0}, \pattern{scale=0.5}{2}{1/2,2/1}{2/1,1/0,2/0,1/2}, \pattern{scale=0.5}{2}{1/2,2/1}{2/1,1/0,2/0,1/1,1/2})_r\;\pi
=
(\pattern{scale=0.5}{2}{1/2,2/1}{2/1,1/0,2/0}, \pattern{scale=0.5}{2}{1/2,2/1}{2/1,1/0,2/0,1/2}, \pattern{scale=0.5}{2}{1/2,2/1}{2/1,1/0,2/0,1/1,1/2})_r\;\psi^{(i_k)}(\pi).
\label{eq:psic}
\end{align}
\end{subequations}
where $(\textrm{pattern})_k$ means the number of the patterns between $\pi(i_{k-1})$ and $\pi(i_k)$.
\end{lemma}

\begin{proof}
If  the pair
$(\pi(j),\pi(i_k))$ with $j<i_k$ contributes the pattern $\pattern{scale=0.5}{2}{1/2,2/1}{2/2,2/0,1/0}$(resp. $\pattern{scale=0.5}{2}{1/2,2/1}{1/1,1/0,2/0,2/2}, \pattern{scale=0.5}{2}{1/2,2/1}{1/1,2/2,2/0,1/0,1/2})$,
then  $j>i_{k-1}$ because $\pi(i_{k-1})<\pi(i_{k})$ and 
$\pi(i)>\pi(i_k)$ for $j\leq i<i_k$.
Also, for  $j<i_k<i$, we have the equivalence 
 $$
  \pi(i_k)<\pi(i)<\pi(j)\Longleftrightarrow \pi(i_k)<\psi_1^{(i_k)}(\pi(j))<\psi_1^{(i_k)}(\pi(i)),
   $$
as $\psi_2^{(i_k)}$ will affect only the letters at the right of 
$\pi(i_{k})$. Thus  we have proved \eqref{eq:psia}.

Next, recall that the operation $\psi^{(i_k)}$ keeps the 
sequence of antirecord positions. The two identities \eqref{eq:psib} and \eqref{eq:psic} are clear if $r<k$.
Assume that $r>k$ and $[\pi]_{>r}=\pi(i_k+1)\ldots \pi(n)$. 
By Lemma~\ref{property:compl} after two operations
$\psi_1^{(i_k)}$ and $\psi_2^{(i_k)}$  the permutation
$\psi^{(i_k)}([\pi]_{>r})$ is  isomorphic with $[\pi]_{>r}$. This proves \eqref{eq:psib} and \eqref{eq:psic}.

\end{proof}

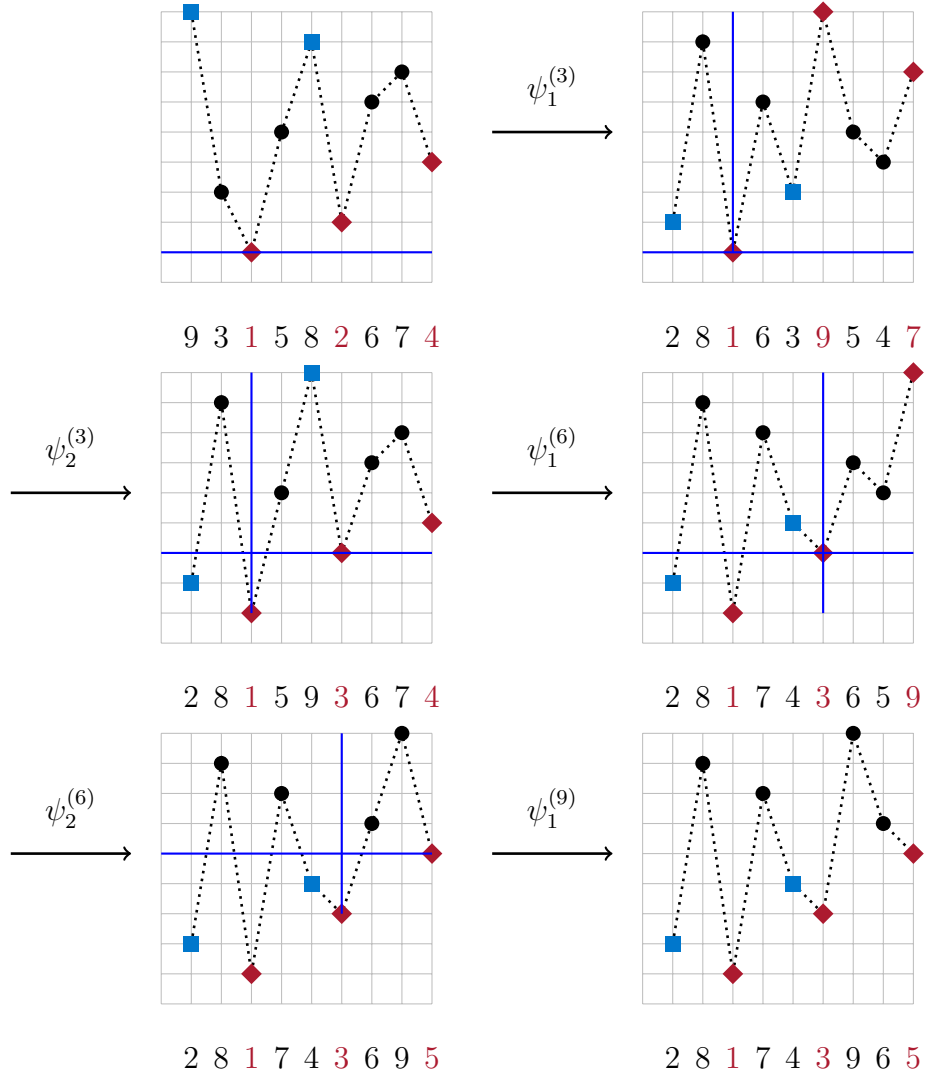
\begin{figure}
\begin{center}
\begin{tikzpicture}[scale=0.4] 	
\draw[step=1,lightgray,thin] (0,0) grid (9,9); 
	\tikzstyle{ridge}=[draw, line width=1, dotted, color=black] 
	\path[ridge] (1,9)--(2,3)--(3,1)--(4,5)--(5,8)--(6,2)--(7,6)--(8,7)--(9,4); 
	\tikzstyle{node0}=[circle, inner sep=2, fill=black] 
	\tikzstyle{node1}=[rectangle, inner sep=3, fill=mhcblue] 
	\tikzstyle{node2}=[diamond, inner sep=2, fill=davidsonred] 
	\node[node1] at (1,9) {}; 
	\node[node0] at (2,3) {}; 
	\node[node2] at (3,1) {}; 
	\node[node0] at (4,5) {}; 
	\node[node1] at (5,8) {}; 
	\node[node2] at (6,2) {}; 
	\node[node0] at (7,6) {}; 
	\node[node0] at (8,7) {}; 
	\node[node2] at (9,4) {}; 
	\draw[thick,blue,-] (0,1) to (9,1);	
	\tikzstyle{pi}=[above=-1] 
	\node[pi] at (1,0) {$9$}; 
	\node[pi] at (2,0) {$3$}; 
         \node[pi, color=davidsonred] at (3,0) {$1$}; 
         \node[pi] at (4,0) {$5$};
	\node[pi] at (5,0) {$8$}; 
 \node[pi, color=davidsonred] at (6,0) {$2$}; 	
	\node[pi] at (7,0) {$6$}; 
	\node[pi] at (8,0) {$7$}; 	
         \node[pi, color=davidsonred] at (9,0) {$4$}; 

	\path[draw,line width=1,->] (11,5)--(15,5); 
	\node[pi] at (13,8){$\psi_1^{(3)}$};
	\begin{scope}[shift={(16,0)}] 
	\draw[step=1,lightgray,thin] (0,0) grid (9,9); 
	\path[ridge] (1,2)--(2,8)--(3,1)--(4,6)--(5,3)--(6,9)--(7,5)--(8,4)--(9,7); 
	\node[node1] at (1,2) {}; 
	\node[node0] at (2,8) {}; 
	\node[node2] at (3,1) {}; 
	\node[node0] at (4,6) {}; 
	\node[node1] at (5,3) {}; 
	\node[node2] at (6,9) {}; 
	\node[node0] at (7,5) {}; 
	\node[node0] at (8,4) {};
	\node[node2] at (9,7) {};
       	\draw[thick,blue,-] (0,1) to (9,1);	
	\draw[thick,blue,-] (3,1) to (3,9);	
       
	\node[pi] at (1,0) {$2$}; 
	\node[pi] at (2,0) {$8$}; 
 \node[pi, color=davidsonred] at (3,0) {$1$}; 
         \node[pi] at (4,0) {$6$};  
        \node[pi] at (5,0) {$3$}; 
	\node[pi, color=davidsonred] at (6,0) {$9$};
	\node[pi] at (7,0) {$5$}; 
	\node[pi] at (8,0) {$4$};
      \node[pi, color=davidsonred] at (9,0) {$7$}; 
	\end{scope}
		\path[draw,line width=1,->] (-5,-7)--(-1,-7); 
	\node[pi] at (-3,-4){$\psi_2^{(3)}$};
	\begin{scope}[shift={(0,-12)}] 
	\draw[step=1,lightgray,thin] (0,0) grid (9,9); 
	\path[ridge] (1,2)--(2,8)--(3,1)--(4,5)--(5,9)--(6,3)--(7,6)--(8,7)--(9,4); 
	\node[node1] at (1,2) {}; 
	\node[node0] at (2,8) {}; 
	\node[node2] at (3,1) {}; 
	\node[node0] at (4,5) {}; 
	\node[node1] at (5,9) {}; 
	\node[node2] at (6,3) {}; 
	\node[node0] at (7,6) {}; 
	\node[node0] at (8,7) {};
	\node[node2] at (9,4) {};
		\draw[thick,blue,-] (0,3) to (9,3);	
	\draw[thick,blue,-] (3,1) to (3,9);			
	
	\node[pi] at (1,0) {$2$}; 
	\node[pi] at (2,0) {$8$}; 
	\node[pi, color=davidsonred] at (3,0) {$1$}; 
      \node[pi] at (4,0) {$5$}; 
	\node[pi] at (5,0) {$9$}; 
	\node[pi, color=davidsonred] at (6,0) {$3$};
	\node[pi] at (7,0) {$6$}; 
	\node[pi] at (8,0) {$7$};
      \node[pi, color=davidsonred] at (9,0) {$4$}; 
	\end{scope}	
	
			\path[draw,line width=1,->] (11,-7)--(15,-7); 
	\node[pi] at (13,-4){$\psi_1^{(6)}$};
	\begin{scope}[shift={(16,-12)}] 
	
		\draw[step=1,lightgray,thin] (0,0) grid (9,9); 
	\path[ridge] (1,2)--(2,8)--(3,1)--(4,7)--(5,4)--(6,3)--(7,6)--(8,5)--(9,9); 
	\node[node1] at (1,2) {}; 
	\node[node0] at (2,8) {}; 
	\node[node2] at (3,1) {}; 
	\node[node0] at (4,7) {}; 
	\node[node1] at (5,4) {}; 
	\node[node2] at (6,3) {}; 
	\node[node0] at (7,6) {}; 
	\node[node0] at (8,5) {};
	\node[node2] at (9,9) {};
		\draw[thick,blue,-] (0,3) to (9,3);	
	\draw[thick,blue,-] (6,1) to (6,9);			
	
	\node[pi] at (1,0) {$2$}; 
	\node[pi] at (2,0) {$8$}; 
	\node[pi, color=davidsonred] at (3,0) {$1$}; 
      \node[pi] at (4,0) {$7$}; 
	\node[pi] at (5,0) {$4$}; 
	\node[pi, color=davidsonred] at (6,0) {$3$};
	\node[pi] at (7,0) {$6$}; 
	\node[pi] at (8,0) {$5$};
      \node[pi, color=davidsonred] at (9,0) {$9$}; 
	\end{scope}		
	
		\path[draw,line width=1,->] (-5,-19)--(-1,-19); 
	\node[pi] at (-3,-16){$\psi_2^{(6)}$};
	\begin{scope}[shift={(0,-24)}] 	
		\draw[step=1,lightgray,thin] (0,0) grid (9,9); 
	\path[ridge] (1,2)--(2,8)--(3,1)--(4,7)--(5,4)--(6,3)--(7,6)--(8,9)--(9,5); 
	\node[node1] at (1,2) {}; 
	\node[node0] at (2,8) {}; 
	\node[node2] at (3,1) {}; 
	\node[node0] at (4,7) {}; 
	\node[node1] at (5,4) {}; 
	\node[node2] at (6,3) {}; 
	\node[node0] at (7,6) {}; 
	\node[node0] at (8,9) {};
	\node[node2] at (9,5) {};
	\draw[thick,blue,-] (0,5) to (9,5);		
	\draw[thick,blue,-] (6,3) to (6,9);			
	
	\node[pi] at (1,0) {$2$}; 
	\node[pi] at (2,0) {$8$}; 
	\node[pi, color=davidsonred] at (3,0) {$1$}; 
      \node[pi] at (4,0) {$7$}; 
	\node[pi] at (5,0) {$4$}; 
	\node[pi, color=davidsonred] at (6,0) {$3$};
	\node[pi] at (7,0) {$6$}; 
	\node[pi] at (8,0) {$9$};
      \node[pi, color=davidsonred] at (9,0) {$5$}; 
	\end{scope}		
\path[draw,line width=1,->] (11,-19)--(15,-19); 
	\node[pi] at (13,-16){$\psi_1^{(9)}$};
	\begin{scope}[shift={(16,-24)}] 
			\draw[step=1,lightgray,thin] (0,0) grid (9,9); 
	\path[ridge] (1,2)--(2,8)--(3,1)--(4,7)--(5,4)--(6,3)--(7,9)--(8,6)--(9,5); 
	\node[node1] at (1,2) {}; 
	\node[node0] at (2,8) {}; 
	\node[node2] at (3,1) {}; 
	\node[node0] at (4,7) {}; 
	\node[node1] at (5,4) {}; 
	\node[node2] at (6,3) {}; 
	\node[node0] at (7,9) {}; 
	\node[node0] at (8,6) {};
	\node[node2] at (9,5) {};
	\node[pi] at (1,0) {$2$}; 
	\node[pi] at (2,0) {$8$}; 
	\node[pi, color=davidsonred] at (3,0) {$1$}; 
      \node[pi] at (4,0) {$7$}; 
	\node[pi] at (5,0) {$4$}; 
	\node[pi, color=davidsonred] at (6,0) {$3$};
	\node[pi] at (7,0) {$9$}; 
	\node[pi] at (8,0) {$6$};
      \node[pi, color=davidsonred] at (9,0) {$5$}; 
	\end{scope}

\end{tikzpicture}
\end{center}

\caption{The involution $\Psi$ on the permutation $931582674$}\label{fig:3}
\end{figure}



\begin{proof}[Proof of Theorem~\ref{main2}]

For  $\pi\in S_n$ and 
$\AREC(\pi)=(i_1,i_2,\ldots, i_l),$ 
 we define the operation  $\Psi$ on $\pi$ by 
 \begin{align}\label{eq:defPsi}
 \Psi(\pi)=\psi^{(i_l)}\circ\cdots\circ \psi^{(i_2)}\circ
 \psi^{(i_1)}(\pi).
 \end{align}

By \eqref{eq:defPsi} the mapping $\Psi$ is reversible with reverse
$$
\Psi^{-1}(\pi)=\psi^{(i_1)}\circ\psi^{(i_2)}\circ\ldots \circ\psi^{(i_l)}(\pi).
$$
Theorem~\ref{main2} follows from Lemma~\ref{lem:opera2}, Lemma~\ref{lemma2:main2} and Lemma~\ref{lem:opera2bis}.
\end{proof}


\begin{example} Figure~\ref{fig:3}.
For $\pi= 9\,3\,1\,5\,8\,2\,6\,7\,4$, we have $\AREC(\pi)=(3, 6,9)$.
 We procced from left to right. 
 
\begin{enumerate} 
\item For position $3$ with value $1$, we have $w=9\,3\,5\, 8\, 2\, 6\, 7\, 4$ and $w^c=2\,8\,6\,3\,9\,5\,4\,7$ Thus  $\psi_1^{(3)}: \pi\mapsto \pi'=2\,8\,1\,6\,3\,9\,5\,4\,7$. 
Next, we have $w=6\,3\, 9\, 5\, 4\, 7$ and $w^c=5\,9\,3\,6\,7\,4$
Thus $\psi_2^{(3)}: \pi'\mapsto \pi''=2\,8\,1\,5\,9\,3\,6\,7\,4$.
\item 
For  position $6$ with value $3$, we have 
$w=5\,9\,6\,7\,4$ and $w^c=7\,4\,6\,5\,9$. So
$\psi_1^{(6)}(\pi'')=2\,8\,1\,7\,4\,3\,6\,5\,9$. Next, we have $w=6\,5\,9$ and $w^c=6\,9\,5$.
Thus we have $\psi^{(6)}(\pi'')=2\,8\,1\,7\,4\,3\,6\,9\,5$.
\item 
For position $9$ with value $5$ we have 
$w=6\,9$ and $w^c=9\,6$. Finally we obtain
$\Psi(\pi)=2\,8\,1\,7\,4\,3\,9\,6\,5$.
\end{enumerate}

Now, we check the mesh patterns.
\begin{itemize}
\item
First, $\psi_1^{(3)}: \pi=9\,3\,1\,5\,8\,2\,6\,7\,4 \mapsto \pi'=2\,8\,1\,6\,3\,9\,5\,4\,7$, the pair $(9, 1)$ contributes the pattern $\pattern{scale=0.5}{2}{1/2,2/1}{2/2,2/0,1/0}$ without the patterns $\pattern{scale=0.5}{2}{1/2,2/1}{1/1,1/0,2/0,2/2}$, $\pattern{scale=0.5}{2}{1/2,2/1}{2/2,1/2,2/0,1/0,1/1}$, 
then the pair $(2,1)$ of $\pi'$ contributes the pattern $\pattern{scale=0.5}{2}{1/2,2/1}{2/1,2/0,1/0}$ without the patterns $\pattern{scale=0.5}{2}{1/2,2/1}{1/2,1/0,2/0,2/1}$, $\pattern{scale=0.5}{2}{1/2,2/1}{2/1,1/2,1/0,2/0,1/1}$, the operations $\psi_2^{(i)}(i=3,6), \psi_1^{(j)}(j=6,9)$ do not change the corresponding mesh pattens at $(2,1)$ of $\pi'$.
\item
Second, $\psi_2^{(3)}\circ \psi_1^{(3)}:\pi=9\,3\,1\,5\,8\,2\,6\,7\,4 \mapsto \pi''=2\,8\,1\,5\,9\,3\,6\,7\,4$,
it is easy to see $5\,8\,2\,6\,7\,4\sim 5\,9\,3\,6\,7\,4$. The pair $(8, 2)$ of $\pi$ contributes the pattern $\pattern{scale=0.5}{2}{1/2,2/1}{2/2,2/0,1/0}$ (resp. $\pattern{scale=0.5}{2}{1/2,2/1}{1/1,1/0,2/0,2/2}$, $\pattern{scale=0.5}{2}{1/2,2/1}{2/2,1/2,2/0,1/0,1/1}$ ),  then the pair $(9, 3)$ of $\pi''$ also contributes the pattern $\pattern{scale=0.5}{2}{1/2,2/1}{2/2,2/0,1/0}$ (resp. $\pattern{scale=0.5}{2}{1/2,2/1}{1/1,1/0,2/0,2/2}$, $\pattern{scale=0.5}{2}{1/2,2/1}{2/2,1/2,2/0,1/0,1/1}$ ),  
 $\psi_1^{(6)}:\pi''=2\,8\,1\,5\,9\,3\,6\,7\,4 \mapsto \pi'''=2\,8\,1\,7\,4\,3\,6\,5\,9$, the pair $(4,3)$ contributes the pattern $\pattern{scale=0.5}{2}{1/2,2/1}{2/1,2/0,1/0}$(resp.  $\pattern{scale=0.5}{2}{1/2,2/1}{1/2,1/0,2/0,2/1}$, $\pattern{scale=0.5}{2}{1/2,2/1}{2/1,1/2,1/0,2/0,1/1}$), the operations $\psi_2^{(6)}, \psi_1^{(9)}$ do not change the corresponding mesh pattens at $(4,3)$ of $\pi'''$.
\end{itemize}

\end{example}

\section{A remark on pattern Nr.\ 14 }

Recall  that an index
$i$ (with $1 < i \leq n$) is a \emph{succession} of $\sigma\in S_n$
 if $\sigma(i) = \sigma(i-1)+1$, see \cite[Section~5]{CHZ97}.
Thus an occurrence of the pattern Nr.\ 14 $= 
\pattern{scale = 0.5}{2}{1/1,2/2}{0/1,1/1,1/2,1/0,1/2,2/1}$ corresponds to
 a succession and  
 we can translate the results on successions in \cite[Section 5]{CHZ97} to
this pattern. For example,
letting 
$$
S_n(x)=\sum_{\pi\in S_n}x^{\pattern{scale = 0.4}{2}{1/1,2/2}{0/1,1/1,1/2,1/0,1/2,2/1}(\pi)}
$$
and  differtentiating the generating function\cite[(5.6)]{CHZ97}
\begin{align}
\sum_{n\geq 0} S_n(x)\frac{t^n}{n!}=\frac{e^{(x-1)t}}{1-t}+(1-x)\int_0^t \frac{e^{(x-1)z}}{1-z}dz
\end{align}
yields 
\begin{align}
\sum_{n\geq 0} S_{n+1}(x)\frac{t^n}{n!}=\frac{e^{(x-1)t}}{(1-t)^2}.
\end{align}
This is the exponential generating function given in \cite[A123513]{OEIS}.
We note that the ordinary generating function (cf. \cite[(5.8)]{CHZ97}) reads
\begin{align}
\sum_{n\geq 0} S_n(x)t^n=\sum_{n\geq 0} \frac{n!t^n}{[1-(x-1)t]^n}.
\end{align}

\section{Acknowledgement}
The first author was supported by the Israel Science Foundation (grant no. 1970/18).

\end{document}